\documentclass{proc-l}

\usepackage{amssymb}

\usepackage{graphicx}

\newtheorem{theorem}{Theorem}[section]
\newtheorem{lemma}[theorem]{Lemma}

\theoremstyle{definition}

\theoremstyle{remark}

\numberwithin{equation}{section}

\begin{document}

\title[$n$-th derivative and fractional integration of Bessel functions]{On the $n$-th derivative and the fractional integration of Bessel
functions with respect to the order}


\author{J. L. Gonz\'{a}lez-Santander}
\address{C/ Ovidi Montllor i Mengual 7, pta. 9. \and 46017, Valencia, Spain.}
\email{juanluis.gonzalezsantander@gmail.com}


\date{}


\commby{}

\begin{abstract}
We obtain integral representations of the $n$-th derivatives of the Bessel
functions with respect to the order. The numerical evaluation of these
expressions is very efficient using a double exponential integration
strategy. Also, from the integral representation corresponding to the
Macdonald function, we have calculated a new integral. Finally, we calculate
integral expressions for the fractional derivatives of the Bessel functions
with respect to the order. Simple proofs for some particular cases given in
the literature are provided as well.
\end{abstract}

\maketitle


\section{Introduction}

Bessel functions are the canonical solutions $y\left( t\right) $ of Bessel's
differential equation:%
\begin{equation}
t^{2}y^{\prime \prime }+t\,y^{\prime }+\left( t^{2}-\nu ^{2}\right) y=0,
\label{Eqn_Bessel}
\end{equation}%
where $\nu $ denotes the order of the Bessel function. This equation arises
when finding separable solutions of Laplace equation in cylindrical
coordinates, and Helmholtz equation in spherical coordinates \cite[Chap. 6]%
{Lebedev}. The general solution of (\ref{Eqn_Bessel})\ is a linear
combination of the \textit{Bessel functions of the first and second kind,
i.e. }$J_{\nu }\left( t\right) $ and $Y_{\nu }\left( t\right) $
respectively. These functions are usually defined as \cite[Eqns. 10.2.2\&3]%
{NIST}%
\begin{equation}
J_{\nu }\left( t\right) =\left( \frac{t}{2}\right) ^{\nu }\sum_{k=0}^{\infty
}\frac{\left( -1\right) ^{k}\left( t/2\right) ^{2k}}{k!\Gamma \left( \nu
+k+1\right) },  \label{Jnu_def}
\end{equation}%
and 
\begin{equation}
Y_{\nu }\left( t\right) =\frac{J_{\nu }\left( t\right) \cos \pi \nu -J_{-\nu
}\left( t\right) }{\sin \pi \nu },\qquad \nu \notin 
\mathbb{Z}
.  \label{Ynu_def}
\end{equation}

In the case of pure imaginary argument, the solutions to the Bessel
equations are called \textit{modified Bessel functions of the first and
second kind, }$I_{\nu }\left( t\right) $ and $K_{\nu }\left( t\right) $
respectively, where \cite[Eqns. 10.25.2\&27.4]{NIST}%
\begin{equation}
I_{\nu }\left( t\right) =\left( \frac{t}{2}\right) ^{\nu }\sum_{k=0}^{\infty
}\frac{\left( t/2\right) ^{2k}}{k!\Gamma \left( \nu +k+1\right) },
\label{Inu_def}
\end{equation}%
and%
\begin{equation}
K_{\nu }\left( t\right) =\frac{\pi }{2}\frac{I_{-\nu }\left( t\right)
-I_{\nu }\left( t\right) }{\sin \pi \nu },\qquad \nu \notin 
\mathbb{Z}
.  \label{Knu_def}
\end{equation}

Despite the fact, the properties of the Bessel functions has been studied
extensively in the literature \cite{Watson, Andrews}, studies about
successive derivatives and repeated integrals of the Bessel functions
with respect to the order $\nu $ are relatively scarce. In the literature,
we find the following series representations \cite[Eqns. 10.15.1\&38.1]{NIST}
for the derivative with respect to the order $\nu \notin 
\mathbb{Z}
$:%
\begin{equation}
\frac{\partial J_{\nu }\left( t\right) }{\partial \nu }=J_{\nu }\left(
t\right) \log \left( \frac{t}{2}\right) -\left( \frac{t}{2}\right) ^{\nu
}\sum_{k=0}^{\infty }\frac{\psi \left( \nu +k+1\right) \left( -1\right)
^{k}\left( t/2\right) ^{2k}}{k!\Gamma \left( \nu +k+1\right) },
\label{DJnu_series}
\end{equation}%
and%
\begin{equation}
\frac{\partial I_{\nu }\left( t\right) }{\partial \nu }=I_{\nu }\left(
t\right) \log \left( \frac{t}{2}\right) -\left( \frac{t}{2}\right) ^{\nu
}\sum_{k=0}^{\infty }\frac{\psi \left( \nu +k+1\right) \left( t/2\right)
^{2k}}{k!\Gamma \left( \nu +k+1\right) },  \label{DInu_series}
\end{equation}%
which are obtained directly from (\ref{Jnu_def})\ and (\ref{Inu_def}). For
the $n$-th derivative of the Bessel function of the first kind with respect
to the order, we find in \cite{Sesma}\ a more complex expression in series
form.

Regarding integral representations of the derivative of $J_{\nu }\left(
t\right) $ and $I_{\nu }\left( t\right) $ with respect to the order, we find
in \cite{Apelblat} $\forall \Re \nu >0$,%
\begin{equation}
\frac{\partial J_{\nu }\left( t\right) }{\partial \nu }=\pi \nu
\int_{0}^{\pi /2}\tan \theta \ Y_{0}\left( t\sin ^{2}\theta \right) J_{\nu
}\left( t\cos ^{2}\theta \right) d\theta ,  \label{DJnu_int_Apelblat}
\end{equation}%
and%
\begin{equation}
\frac{\partial I_{\nu }\left( t\right) }{\partial \nu }=-2\nu \int_{0}^{\pi
/2}\tan \theta \ K_{0}\left( t\sin ^{2}\theta \right) I_{\nu }\left( t\cos
^{2}\theta \right) d\theta .  \label{DInu_int_Apelblat}
\end{equation}

Other integral representations of the order derivative of $J_{\nu }\left(
z\right) $ and $Y_{\nu }\left( z\right) $ are given in \cite{Dunster} for $%
\nu >0$ and $t\neq 0$, $\left\vert \mathrm{arg}\ t\right\vert \leq \pi $,
which read as,%
\begin{equation}
\frac{\partial J_{\nu }\left( t\right) }{\partial \nu }=\pi \nu \left[
Y_{\nu }\left( t\right) \int_{0}^{t}\frac{J_{\nu }^{2}\left( z\right) }{t}%
dz+J_{\nu }\left( t\right) \int_{t}^{\infty }\frac{J_{\nu }\left( z\right)
Y_{\nu }\left( z\right) }{z}dz\right] ,  \label{DJnu_int_Dunster}
\end{equation}%
and%
\begin{eqnarray}
&&\frac{\partial Y_{\nu }\left( t\right) }{\partial \nu }
\label{DYnu_int_Dunster} \\
&=&\pi \nu \left[ J_{\nu }\left( t\right) \left( \int_{t}^{\infty }\frac{%
Y_{\nu }^{2}\left( z\right) }{z}dz-\frac{1}{2\nu }\right) -Y_{\nu }\left(
t\right) \int_{t}^{\infty }\frac{J_{\nu }\left( z\right) Y_{\nu }\left(
z\right) }{z}dz\right] .  \notag
\end{eqnarray}

Recently, in \cite{DBesselJL}, we find the following integral
representations of the derivatives of the modified Bessel functions $I_{\nu
}\left( t\right) $ and $K_{\nu }\left( t\right) $ with respect to the order
for $\nu >0$ and $t\neq 0$,$\left\vert \mathrm{arg}\ t\right\vert \leq \pi $,%
\begin{equation}
\frac{\partial I_{\nu }\left( t\right) }{\partial \nu }=-2\nu \left[ I_{\nu
}\left( t\right) \int_{t}^{\infty }\frac{K_{\nu }\left( z\right) I_{\nu
}\left( z\right) }{z}dz+K_{\nu }\left( t\right) \int_{0}^{t}\frac{I_{\nu
}^{2}\left( z\right) }{z}dz\right] ,  \label{DInu_int}
\end{equation}%
and%
\begin{equation}
\frac{\partial K_{\nu }\left( t\right) }{\partial \nu }=2\nu \left[ K_{\nu
}\left( t\right) \int_{t}^{\infty }\frac{I_{\nu }\left( z\right) K_{\nu
}\left( z\right) }{z}dz-I_{\nu }\left( t\right) \int_{t}^{\infty }\frac{%
K_{\nu }^{2}\left( z\right) }{z}dz\right] .  \label{DKnu_int}
\end{equation}

The great advantage of the integral expressions (\ref{DJnu_int_Dunster})-(%
\ref{DKnu_int})\ is that the integrals involved in them can be calculated in
closed-form \cite{DBesselJL}. Also, expressions in closed-form for the
second and third derivatives with respect to the order are found in \cite%
{BrychovNew}, but these expressions are extraordinarily complex, above all
for the third derivative.

Regarding integration of Bessel functions with respect to the order, the
results found in the literature are even more scarce. For instance, in \cite%
{Apelblat}, we find 
\begin{eqnarray}
\int_{\nu }^{\infty }J_{\mu }\left( t\right) d\mu &=&\frac{1}{2}+\frac{1}{%
\pi }\int_{0}^{\pi }\sin \left( t\sin x-\nu x\right) \frac{dx}{x}
\label{Appelblat_J} \\
&&-\frac{1}{\pi }\int_{0}^{\infty }\frac{e^{-t\sinh x-\nu x}}{\pi ^{2}+x^{2}}%
\left( x\sin \pi \nu +\cos \pi \nu \right) dx,  \notag
\end{eqnarray}%
and 
\begin{eqnarray}
\int_{\nu }^{\infty }I_{\mu }\left( t\right) d\mu &=&\frac{e^{t}}{2}-\frac{1%
}{\pi }\int_{0}^{\pi }e^{t\cos x}\sin \nu x\frac{dx}{x}  \label{Appelblat_I}
\\
&&-\frac{1}{\pi }\int_{0}^{\infty }\frac{e^{-t\cosh x-\nu x}}{\pi ^{2}+x^{2}}%
\left( x\sin \pi \nu +\cos \pi \nu \right) dx,  \notag
\end{eqnarray}%
which are calculated using the complex contour integration of a particular inverse Laplace
transform \cite[Sect. 88]{Churchill}. In the Appendix, we provide simples
proofs of (\ref{Appelblat_J}) and (\ref{Appelblat_I}) by direct integration.
Notice that a second integration with respect to the order in (\ref%
{Appelblat_J}) or (\ref{Appelblat_I}), taking again as integration interval $%
\left( \nu ,\infty \right) $, would be divergent. However, repeated
integration of Bessel functions with respect to the order is possible if we
take a finite integration interval. It is worth noting that the latter seems
to be absent in the most common literature.

Therefore, the goal of this article is two-folded. On the one hand, in
Section \ref{Section: Derivatives}, we obtain simple integral
representations for the $n$-th derivatives of the Bessel functions with
respect to the order. The great advantage of these expressions is that its
numerical evaluation is quite rapid and straightforward. As a by-product, we
obtain the calculation of an integral which does not seem to be reported in
the literature. Also, the values of the integral representations obtained at
argument $t=0$ are calculated.

On the other hand, in Section \ref{Section: Fractional Integration}, we
calculate the iterated integrals of Bessel functions, using fractional
integration. Also, values at argument $t=0$ are calculated. It is worth
noting that the latter is not trivial from the integral representations
obtained.

Finally, we collect our conclusions in Section \ref{Section: Conclusions}.

\section{Integral representations of $n$-th order derivatives\label{Section:
Derivatives}}

In order to perform the $n$-th derivatives of Bessel and modified Bessel
functions with respect to the order, first we state the following $n$-th
derivatives, that can be proved easily by induction.

\begin{lemma}
The $n$-th derivative of the functions 
\begin{eqnarray*}
f_{1}\left( \nu \right) &=&\cos \left( t\sin x-\nu x\right) , \\
f_{2}\left( \nu \right) &=&\sin \left( t\sin x-\nu x\right) , \\
f_{3}\left( \nu \right) &=&e^{-\nu x}\sin \pi \nu =\Im \left( e^{\left( i\pi
-x\right) \nu }\right) , \\
f_{4}\left( \nu \right) &=&e^{-\nu x}\cos \pi \nu =\Re \left( e^{\left( i\pi
-x\right) \nu }\right) ,
\end{eqnarray*}%
with respect to the order $\nu $ are given by%
\begin{eqnarray}
\ f_{1}^{\left( n\right) }\left( \nu \right) &=&x^{n}\cos \left( t\sin x-\nu
x-n\pi /2\right) ,  \label{f1(n)} \\
\ f_{2}^{\left( n\right) }\left( \nu \right) &=&x^{n}\sin \left( t\sin x-\nu
x-n\pi /2\right) ,  \label{f2(n)} \\
f_{3}^{\left( n\right) }\left( \nu \right) &=&e^{-\nu x}\Im \left[ \left(
i\pi -x\right) ^{n}e^{i\pi \nu }\right] ,  \label{f3(n)} \\
f_{4}^{\left( n\right) }\left( \nu \right) &=&e^{-\nu x}\Re \left[ \left(
i\pi -x\right) ^{n}e^{i\pi \nu }\right] .  \label{f4(n)}
\end{eqnarray}
\end{lemma}

To obtain the $n$-th derivative of the Bessel function of the first kind
with respect to the order, we depart from the Schl\"{a}fli integral
representation of $J_{\nu }\left( t\right) $ \cite[Eqn. 10.9.6]{NIST},
wherein we have $\forall \Re t>0$, 
\begin{equation}
J_{\nu }\left( t\right) =\frac{1}{\pi }\int_{0}^{\pi }\cos \left( t\sin
x-\nu x\right) dx-\frac{\sin \nu \pi }{\pi }\int_{0}^{\infty }e^{-t\sinh
x-\nu x}dx.  \label{Jnu_Int}
\end{equation}

Therefore, applying (\ref{f1(n)})\ and (\ref{f3(n)}), the $n$-th derivative
of $J_{\nu }\left( t\right) $ with respect to the order is%
\begin{eqnarray}
\frac{\partial ^{n}}{\partial \nu ^{n}}J_{\nu }\left( t\right) &=&\frac{1}{%
\pi }\int_{0}^{\pi }x^{n}\cos \left( t\sin x-\nu x-\frac{\pi }{2}n\right) dx
\label{Dn_Jnu} \\
&&-\frac{1}{\pi }\int_{0}^{\infty }e^{-t\sinh x-\nu x}\Im \left[ \left( i\pi
-x\right) ^{n}e^{i\pi \nu }\right] dx.  \notag
\end{eqnarray}

Similar calculations can be carried out for the Bessel function of the
second kind, whose integral representation is \cite[Eqn. 10.9.7]{NIST}, $%
\forall \Re t>0$, 
\begin{eqnarray}
Y_{\nu }\left( t\right) &=&\frac{1}{\pi }\int_{0}^{\pi }\sin \left( t\sin
x-\nu x\right) dx  \label{Ynu_int} \\
&&-\frac{1}{\pi }\int_{0}^{\infty }e^{-t\sinh x}\left( e^{\nu x}+e^{-\nu
x}\cos \nu \pi \right) dx.  \notag
\end{eqnarray}

Therefore, according to (\ref{f2(n)})\ and (\ref{f4(n)}), we obtain 
\begin{eqnarray}
&&\frac{\partial ^{n}}{\partial \nu ^{n}}Y_{\nu }\left( t\right)
\label{Dn_Ynu} \\
&=&\frac{1}{\pi }\int_{0}^{\pi }x^{n}\sin \left( t\sin x-\nu x-\frac{\pi }{2}%
n\right) dx  \notag \\
&&-\frac{1}{\pi }\int_{0}^{\infty }e^{-t\sinh x}\left( x^{n}e^{\nu
x}-e^{-\nu x}\Re \left[ \left( i\pi -x\right) ^{n}e^{i\pi \nu }\right]
\right) dx.  \notag
\end{eqnarray}

For the modified Bessel function, we find in the literature the following
integral representation \cite[Eqn. 10.32.4]{NIST}, $\forall \Re t>0$, 
\begin{equation}
I_{\nu }\left( t\right) =\frac{1}{\pi }\int_{0}^{\pi }e^{t\cos x}\cos \nu x\
dx-\frac{\sin \nu \pi }{\pi }\int_{0}^{\infty }e^{-t\cosh x-\nu x}dx.
\label{Inu_int}
\end{equation}

Therefore, according to (\ref{f1(n)}) with $t=0$\ and (\ref{f3(n)}), we have%
\begin{eqnarray}
\frac{\partial ^{n}}{\partial \nu ^{n}}I_{\nu }\left( t\right) &=&\frac{1}{%
\pi }\int_{0}^{\pi }x^{n}e^{t\cos x}\cos \left( \nu x+\frac{\pi }{2}n\right)
dx  \label{Dn_Inu} \\
&&-\frac{1}{\pi }\int_{0}^{\infty }e^{-t\cosh x-\nu x}\Im \left[ \left( i\pi
-x\right) ^{n}e^{i\pi \nu }\right] dx.  \notag
\end{eqnarray}

Also, from the integral representation of the Macdonald function \cite[Eqn.
5.10.23]{Lebedev}, $\forall \Re t>0$, 
\begin{equation}
K_{\nu }\left( t\right) =\frac{1}{2}\int_{-\infty }^{\infty }e^{\nu x-t\cosh
x}\ dx,  \label{Knu_int}
\end{equation}%
we have%
\begin{equation}
\frac{\partial ^{n}}{\partial \nu ^{n}}K_{\nu }\left( t\right) =\frac{1}{2}%
\int_{-\infty }^{\infty }x^{n}e^{\nu x-t\cosh x}dx.  \label{Dn_Knu}
\end{equation}

For $n=1$, the above integral (\ref{Dn_Knu})\ is calculated in \cite%
{DBesselJL}\ in closed-form, thus $\forall \nu >0$, $\Re t>0$, we have%
\begin{eqnarray}
&&\int_{-\infty }^{\infty }x\,e^{\nu x-t\cosh x}dx  \label{int_new_G} \\
&=&\nu \left[ \frac{K_{\nu }\left( z\right) }{\sqrt{\pi }}%
G_{2,4}^{3,1}\left( z^{2}\left\vert 
\begin{array}{c}
1/2,1 \\ 
0,0,\nu ,-\nu%
\end{array}%
\right. \right) -\sqrt{\pi }I_{\nu }\left( z\right) G_{2,4}^{4,0}\left(
z^{2}\left\vert 
\begin{array}{c}
1/2,1 \\ 
0,0,\nu ,-\nu%
\end{array}%
\right. \right) \right] ,  \notag
\end{eqnarray}%
where $G_{p,q}^{m,n}$ denotes the Meijer-$G$ function \cite[%
Eqn. 16.17.1]{NIST}. If $\nu \notin 
\mathbb{Z}
$, the above expression is reduced in terms of generalized hypergeometric
functions $_{p}F_{q}$ \cite[Eqn. 16.2.1]{NIST} as \cite{BrychovNew},%
\begin{eqnarray}
&&\int_{-\infty }^{\infty }x\,e^{\nu x-t\cosh x}dx  \label{int_new_F} \\
&=&\pi \csc \pi \nu \left\{ \pi \cot \pi \nu \,I_{\nu }\left( z\right) - 
\left[ I_{\nu }\left( z\right) +I_{-\nu }\left( z\right) \right] 
\begin{array}{c}
\displaystyle
\\ 
\displaystyle%
\end{array}%
\right.  \notag \\
&&\left. \left[ \frac{z^{2}}{4\left( 1-\nu ^{2}\right) }\,_{3}F_{4}\left(
\left. 
\begin{array}{c}
1,1,\frac{3}{2} \\ 
2,2,2-\nu ,2+\nu%
\end{array}%
\right\vert z^{2}\right) +\log \left( \frac{z}{2}\right) -\psi \left( \nu
\right) -\frac{1}{2\nu }\right] \right\}  \notag \\
&&+\frac{1}{2}\left\{ I_{-\nu }\left( z\right) \Gamma ^{2}\left( -\nu
\right) \left( \frac{z}{2}\right) ^{2\nu }\,_{2}F_{3}\left( \left. 
\begin{array}{c}
\nu ,\frac{1}{2}+\nu \\ 
1+\nu ,1+\nu ,1+2\nu%
\end{array}%
\right\vert z^{2}\right) \right.  \notag \\
&&\quad -\left. I_{\nu }\left( z\right) \Gamma ^{2}\left( \nu \right) \left( 
\frac{z}{2}\right) ^{-2\nu }\,_{2}F_{3}\left( \left. 
\begin{array}{c}
-\nu ,\frac{1}{2}-\nu \\ 
1-\nu ,1-\nu ,1-2\nu%
\end{array}%
\right\vert z^{2}\right) \right\} .  \notag
\end{eqnarray}

It is worth noting that the numerical evaluation of the integral
representations for the $n$-th derivatives of the Bessel functions, i.e. (%
\ref{Dn_Jnu}), (\ref{Dn_Ynu}), (\ref{Dn_Inu}), and (\ref{Dn_Knu}), is very
efficient if we use a \textit{double exponential}\ strategy \cite{TakahasiMori}.
This is especially significant when $n\geq 3$. In order to have an idea of
how quick is the performance of the formulas presented here, define $t_{%
\text{int}}$ as the timing of the numerical evaluation of integral
representations (\ref{Dn_Jnu}), (\ref{Dn_Ynu}), (\ref{Dn_Inu}), and (\ref%
{Dn_Knu});\ and $t_{\text{num}}$ as the timing of the numerical method
provided by MATHEMATICA\ to compute derivatives, thus the timing ratio of both methods is 
$\chi =t_{\text{num}}/t_{\text{int}}$. For instance, Fig. \ref{Figure: derivada 3 J}
shows the plot of $\partial ^{3}/\partial \nu ^{3}J_{\nu }\left( t\right) $
in the domain $\left( \nu ,t\right) \in \left( 0,10\right) \times \left(
0,10\right) $, and the timing ratio in this case is $\chi =t_{\text{num}}/t_{%
\text{int}}\approx 35$. Similar ratios are obtained for the other $n$-th
derivatives integral expressions, and the higher is $n$, the higher is the
timing ratio $\chi $. 

\begin{figure}
\includegraphics{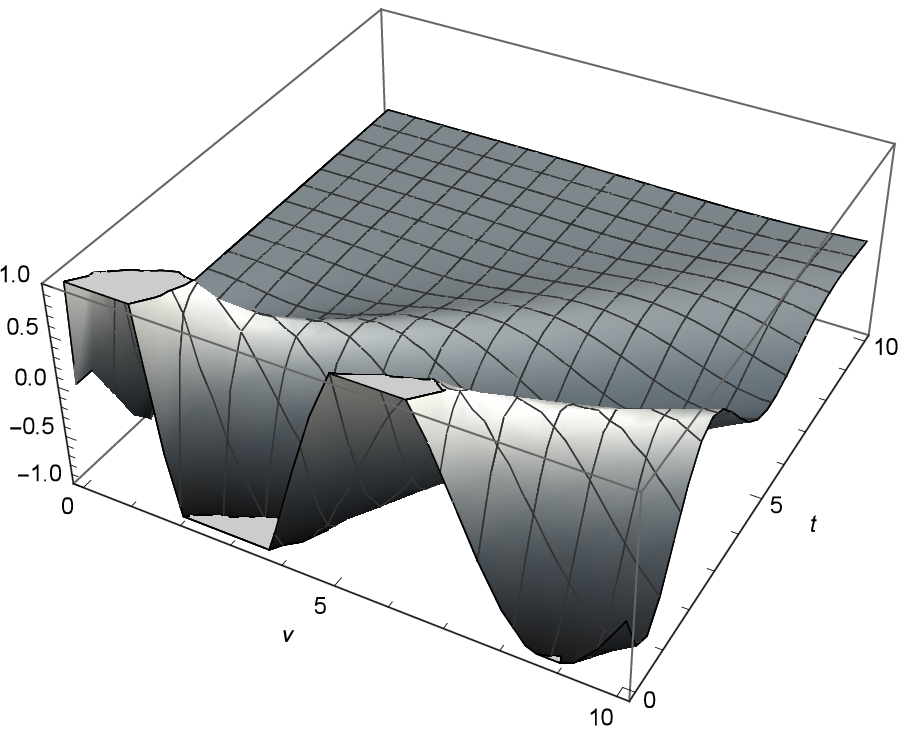}
\caption{Graph of $\frac{\partial ^{3}}{\partial 
\protect\nu ^{3}}J_{\protect\nu }\left( t\right) $ in the domain $\left( 
\protect\nu ,t\right) \in \left( 0,10\right) \times \left( 0,10\right) $.}
\label{Figure: derivada 3 J}
\end{figure}

\subsection{Values at $t=0$}

On the one hand, the derivatives of the irregular solutions of the Bessel's equation 
are infinite $\forall t=0$. Indeed, from (\ref{Dn_Ynu}), we have $%
\forall \nu \geq 0$, $n=0,1,\ldots $%
\begin{equation*}
\lim_{t\rightarrow 0^{+}}\frac{\partial ^{n}}{\partial \nu ^{n}}Y_{\nu
}\left( t\right) =-\infty ,
\end{equation*}%
and from (\ref{Dn_Knu}), 
\begin{equation*}
\lim_{t\rightarrow 0^{+}}\frac{\partial ^{n}}{\partial \nu ^{n}}K_{\nu
}\left( t\right) =\infty .
\end{equation*}

On the other hand, for the regular solutions, the derivatives with respect
to order $\forall t=0$ are null. Indeed, according to the series
representations (\ref{Jnu_def})\ and (\ref{Inu_def}), we have%
\begin{equation}
J_{\nu }\left( 0\right) =I_{\nu }\left( 0\right) =\left\{ 
\begin{array}{cc}
1, & \nu =0 \\ 
0, & \nu >0%
\end{array}%
\right.  \label{Jnu(0) = Inu(0) = 0}
\end{equation}

Therefore, $\forall \nu \geq 0$, $n=1,2,\ldots $ 
\begin{equation}
\frac{\partial ^{n}}{\partial \nu ^{n}}J_{\nu }\left( 0\right) =\frac{%
\partial ^{n}}{\partial \nu ^{n}}I_{\nu }\left( 0\right) =0.
\label{Dnu_J(0) = Dnu_I(0)}
\end{equation}

However, (\ref{Dnu_J(0) = Dnu_I(0)})\ is not obvious from the integral
representations (\ref{Jnu_Int}) and (\ref{Inu_int}). Next, we provide a
simple proof of (\ref{Dnu_J(0) = Dnu_I(0)}) from this point of view.

\bigskip

\begin{proof}
Rewrite (\ref{Dn_Jnu}) and (\ref{Dn_Inu})\ as%
\begin{eqnarray*}
\frac{\partial ^{n}}{\partial \nu ^{n}}J_{\nu }\left( t\right) &=&\frac{1}{%
\pi }\Re \int_{0}^{\pi }e^{it\sin x}\left( ix\right) ^{n}e^{i\nu x}dx \\
&&-\frac{1}{\pi }\Im \int_{0}^{\infty }e^{-t\sinh x}\left( i\pi -x\right)
^{n}e^{\left( i\pi -x\right) \nu }dx.
\end{eqnarray*}%
and%
\begin{eqnarray*}
\frac{\partial ^{n}}{\partial \nu ^{n}}I_{\nu }\left( t\right) &=&\frac{1}{%
\pi }\Re \int_{0}^{\pi }e^{t\cos x}\left( ix\right) ^{n}e^{i\nu x}dx \\
&&-\frac{1}{\pi }\Im \int_{0}^{\infty }e^{-t\cosh x}\left( i\pi -x\right)
^{n}e^{\left( i\pi -x\right) \nu }dx.
\end{eqnarray*}%
Therefore,%
\begin{eqnarray}
\frac{\partial ^{n}}{\partial \nu ^{n}}J_{\nu }\left( 0\right) &=&\frac{%
\partial ^{n}}{\partial \nu ^{n}}I_{\nu }\left( 0\right)  \notag \\
&=&\frac{1}{\pi }\Re \int_{0}^{\pi }\left( ix\right) ^{n}e^{i\nu x}dx  \notag
\\
&&-\frac{1}{\pi }\Im \int_{0}^{\infty }\left( i\pi -x\right) ^{n}e^{\left(
i\pi -x\right) \nu }dx.  \label{Im(iPi-x)}
\end{eqnarray}%
Note that, performing the substitution $z=i\pi -x$ in (\ref{Im(iPi-x)}) and
splitting the integration path in real and pure imaginary parts, we have $%
\forall \nu >0$%
\begin{eqnarray*}
\Im \int_{0}^{\infty }\left( i\pi -x\right) ^{n}e^{\left( i\pi -x\right) \nu
}dx &=&\Im \int_{-\infty }^{i\pi }z^{n}e^{z\nu }dz \\
&=&\Im \left[ \int_{-\infty }^{0}z^{n}e^{z\nu }dz+\int_{0}^{i\pi
}z^{n}e^{z\nu }dz\right] \\
&=&\Im \int_{0}^{i\pi }z^{n}e^{z\nu }dz.
\end{eqnarray*}%
Now, perform the substitution $z=ix$ and apply the property $\Im \left(
iz\right) =\Re \left( z\right) $ to arrive at%
\begin{equation}
\Im \int_{0}^{\infty }\left( i\pi -x\right) ^{n}e^{\left( i\pi -x\right) \nu
}dx=\Re \int_{0}^{\pi }\left( ix\right) ^{n}e^{i\nu x}dx.
\label{Im(iPi-x)_resultado}
\end{equation}%
Inserting (\ref{Im(iPi-x)_resultado})\ in (\ref{Im(iPi-x)}), the proof is
completed.\bigskip
\end{proof}

\section{Repeated integration with respect to the order\label{Section:
Fractional Integration}}

Here, we will adopt the following notation for the repeated integration, 
\begin{equation*}
D_{x-x_{0}}^{-n}f\left( x\right)
=\int_{x_{0}}^{x}\int_{x_{0}}^{t_{n-1}}\cdots \int_{x_{0}}^{t_{1}}f\left(
t_{0}\right) dt_{0}\cdots dt_{n-1}.
\end{equation*}

According to "Cauchy's iterated integral" \cite{Lavoie}, we have%
\begin{equation*}
D_{x-x_{0}}^{-n}f\left( x\right) =\frac{1}{\left( n-1\right) !}%
\int_{x_{0}}^{x}\left( x-t\right) ^{n-1}f\left( t\right) dt.
\end{equation*}

We can generalize the above result replacing $n$ by $\alpha $, 
\begin{equation}
D_{x-x_{0}}^{-\alpha }f\left( x\right) =\frac{1}{\Gamma \left( \alpha
\right) }\int_{x_{0}}^{x}\left( x-t\right) ^{\alpha -1}f\left( t\right)
dt,\qquad \Re \alpha >0.  \label{Int_frac_def}
\end{equation}

When $x_{0}=0$, (\ref{Int_frac_def})\ is called the Riemann-Liouville
integral. Thereby, $D_{x-x_{0}}^{-\alpha }$ denotes the fractional
integration of order $\alpha $. In order to perform the fractional
integration of the Bessel functions with respect to the order, we need the
following result.

\begin{lemma}
The fractional integration of the exponential function is given by%
\begin{equation}
D_{x-x_{0}}^{-\alpha }e^{sx}=\frac{e^{sx}}{s^{\alpha }}P\left( \alpha
,s\left( x-x_{0}\right) \right) ,  \label{Int_frac_Exp}
\end{equation}%
where 
\begin{equation}
P\left( a,z\right) =\frac{1}{\Gamma \left( \alpha \right) }%
\int_{0}^{z}t^{a-1}e^{-t}dt,\qquad \Re \alpha >0,  \label{gamma_low_def}
\end{equation}%
denotes the regularized lower incomplete gamma function \cite[Eqn. 8.2.4]%
{NIST}.
\end{lemma}

\begin{proof}
Taking in (\ref{Int_frac_def})\ $f\left( x\right) =e^{sx}$, and performing
the change of variables $v=s\left( x-t\right) $, we arrive at 
\begin{equation*}
D_{x-x_{0}}^{-\alpha }e^{sx}=\frac{e^{sx}}{\Gamma \left( \alpha \right)
s^{\alpha }}\int_{0}^{s\left( x-x_{0}\right) }v^{\alpha -1}e^{-v}dv,
\end{equation*}%
which, according to (\ref{gamma_low_def}), is the desired result given in (%
\ref{Int_frac_Exp}).
\end{proof}

\bigskip

According to the integral representation (\ref{Jnu_Int}), and taking into
account (\ref{Int_frac_Exp}), the fractional integration of the Bessel
function of the first kind with respect to the order is, $\forall t>0$,%
\begin{eqnarray}
&&D_{\nu -\nu _{0}}^{-\alpha }J_{\nu }\left( t\right)  \label{DJnu_a} \\
&=&\frac{1}{\pi \,}\left\{ \int_{0}^{\pi }x^{-\alpha }\Re \left( e^{i\left(
t\sin x-\nu x+\pi \alpha /2\right) }P\left( \alpha ,-ix\left( \nu -\nu
_{0}\right) \right) \right) dx\right.  \notag \\
&&-\left. \int_{0}^{\infty }e^{-t\sinh x}\Im \left( e^{\left( i\pi -x\right)
\nu }\frac{P\left( \alpha ,\left( i\pi -x\right) \left( \nu -\nu _{0}\right)
\right) }{\left( i\pi -x\right) ^{\alpha }}\right) dx\right\} .  \notag
\end{eqnarray}

Similarly, from (\ref{Inu_int}) and taking into account (\ref{Int_frac_Exp}%
), the fractional integration of the modified Bessel function is, $\forall
\Re t>0$, 
\begin{eqnarray}
&&D_{\nu -\nu _{0}}^{-\alpha }I_{\nu }\left( t\right)  \label{DInu_a} \\
&=&\frac{1}{\pi \,}\left\{ \int_{0}^{\pi }e^{t\cos x}x^{-\alpha }\Re \left(
e^{i\left( \nu x-\pi \alpha /2\right) }P\left( \alpha ,ix\left( \nu -\nu
_{0}\right) \right) \right) dx\right.  \notag \\
&&-\left. \int_{0}^{\infty }e^{-t\cosh x}\Im \left( e^{\left( i\pi -x\right)
\nu }\frac{P\left( \alpha ,\left( i\pi -x\right) \left( \nu -\nu _{0}\right)
\right) }{\left( i\pi -x\right) ^{\alpha }}\right) dx\right\} .  \notag
\end{eqnarray}

Also, from (\ref{Ynu_int}), $\forall t>0$, 
\begin{eqnarray}
&&D_{\nu -\nu _{0}}^{-\alpha }Y_{\nu }\left( t\right)  \label{DYnu_a} \\
&=&\frac{1}{\pi \,}\left\{ \int_{0}^{\pi }x^{-\alpha }\Im \left( e^{i\left(
t\sin x-\nu x+\pi \alpha /2\right) }P\left( \alpha ,-ix\left( \nu -\nu
_{0}\right) \right) \right) dx\right.  \notag \\
&&-\int_{0}^{\infty }e^{-t\sinh x}\left( e^{\nu x}\frac{P\left( \alpha
,x\left( \nu -\nu _{0}\right) \right) }{x^{\alpha }}\right.  \notag \\
&&\quad +\left. \left. \Re \left[ e^{\left( i\pi -x\right) \nu }\frac{%
P\left( \alpha ,\left( i\pi -x\right) \left( \nu -\nu _{0}\right) \right) }{%
\left( i\pi -x\right) ^{\alpha }}\right] \right) dx\right\} .  \notag
\end{eqnarray}

Finally, from the integral representation of the Macdonald function (\ref%
{Knu_int}), we obtain a very simple expression for the fractional
integration $\forall t>0$, 
\begin{equation}
D_{\nu -\nu _{0}}^{-\alpha }K_{\nu }\left( t\right) =\frac{1}{2}%
\int_{-\infty }^{\infty }x^{-\alpha }\Re \left[ e^{\nu x-t\cosh x}P\left(
\alpha ,x\left( \nu -\nu _{0}\right) \right) \right] dx,  \label{DKnu_a}
\end{equation}%
wherein we have taken the real part of the integral because, according to (%
\ref{Int_frac_def}), $D_{\nu -\nu _{0}}^{-\alpha }K_{\nu }\left( t\right) $
is real $\forall t>0$, thus the imaginary part of (\ref{DKnu_a}) must vanish.

It is worth noting that the finite integrals given in (\ref{DJnu_a})-(\ref%
{DYnu_a})\ are efficiently evaluated using the numerical \textit{double exponential} strategy \cite{TakahasiMori}.

\subsection{Values at $t=0$}

Note that from (\ref{Ynu_int}), we have that $\lim_{t\rightarrow
0^{+}}Y_{\nu }\left( t\right) =-\infty $, thus, according to (\ref%
{Int_frac_def}), we have 
\begin{equation*}
\lim_{t\rightarrow 0^{+}}D_{\nu -\nu _{0}}^{-\alpha }Y_{\nu }\left( t\right)
=-\infty .
\end{equation*}

Also, from (\ref{Knu_int}), we have that $\lim_{t\rightarrow 0^{+}}K_{\nu
}\left( t\right) =\infty $, thus 
\begin{equation*}
\lim_{t\rightarrow 0^{+}}D_{\nu -\nu _{0}}^{-\alpha }K_{\nu }\left( t\right)
=\infty .
\end{equation*}

Also, according to (\ref{Jnu(0) = Inu(0) = 0}) and (\ref{Int_frac_def}), we
have%
\begin{equation}
D_{\nu -\nu _{0}}^{-\alpha }J_{\nu }\left( 0\right) =D_{\nu -\nu
_{0}}^{-\alpha }I_{\nu }\left( 0\right) =0.  \label{DJ(0) = DI(0)}
\end{equation}

However, (\ref{DJ(0) = DI(0)}) is not obvious from the integral
representations given in (\ref{DJnu_a})\ and (\ref{DInu_a}). Next, we
provide a simple proof of the latter.

\begin{proof}
Substitute $t=0$ in (\ref{DJnu_a}) and (\ref{DInu_a}), and rewrite the
result as follows,%
\begin{eqnarray}
&&D_{\nu -\nu _{0}}^{-\alpha }J_{\nu }\left( 0\right) =D_{\nu -\nu
_{0}}^{-\alpha }I_{\nu }\left( 0\right)  \label{Int_frac_J0_1} \\
&=&\frac{1}{\pi \,}\left\{ \Re \int_{0}^{\pi }e^{\pm i\nu x}\frac{P\left(
\alpha ,\pm ix\left( \nu -\nu _{0}\right) \right) }{\left( \pm ix\right)
^{\alpha }}dx\right.  \notag \\
&&-\left. \Im \int_{0}^{\infty }e^{\left( i\pi -x\right) \nu }\frac{P\left(
\alpha ,\left( i\pi -x\right) \left( \nu -\nu _{0}\right) \right) }{\left(
i\pi -x\right) ^{\alpha }}dx\right\} .  \notag
\end{eqnarray}%
Notice that the first integral on the RHS\ of (\ref{Int_frac_J0_1}) is
independent of the `$\pm $' sign since $\forall a\in 
\mathbb{R}
$, $\overline{P\left( a,z\right) }=P\left( a,\bar{z}\right) $ and $\Re
\left( z\right) =\Re \left( \bar{z}\right) $. Perform the change of
variables $z=i\pi -x$ in the second integral on the RHS\ of (\ref%
{Int_frac_J0_1}) to obtain%
\begin{eqnarray}
&&\Im \int_{0}^{\infty }e^{\left( i\pi -x\right) \nu }\frac{P\left( \alpha
,\left( i\pi -x\right) \left( \nu -\nu _{0}\right) \right) }{\left( i\pi
-x\right) ^{\alpha }}dx  \notag \\
&=&\Im \int_{-\infty }^{i\pi }e^{z\nu }\frac{P\left( \alpha ,z\left( \nu
-\nu _{0}\right) \right) }{z^{\alpha }}dz  \notag \\
&=&\Im \int_{-\infty }^{0}e^{z\nu }\frac{P\left( \alpha ,z\left( \nu -\nu
_{0}\right) \right) }{z^{\alpha }}dz+\Im \int_{0}^{i\pi }e^{z\nu }\frac{%
P\left( \alpha ,z\left( \nu -\nu _{0}\right) \right) }{z^{\alpha }}dz.
\label{Int_Frac_J0_2}
\end{eqnarray}%
Note that the first integral in (\ref{Int_Frac_J0_2})\ vanishes. Perform the
substitution $z=ix$ in the second integral to arrive at%
\begin{equation*}
\Im \int_{0}^{\pi }ie^{i\nu x}\frac{P\left( \alpha ,ix\left( \nu -\nu
_{0}\right) \right) }{\left( ix\right) ^{\alpha }}dx.
\end{equation*}%
Taking into account that $\forall a\in 
\mathbb{R}
$, $\overline{P\left( a,z\right) }=P\left( a,\bar{z}\right) $, and the
property $\Im \left( i\bar{z}\right) =\Re \left( z\right) =\Re \left( \bar{z}%
\right) $, we finally get%
\begin{eqnarray}
&&\Im \int_{0}^{\infty }e^{\left( i\pi -x\right) \nu }\frac{P\left( \alpha
,\left( i\pi -x\right) \left( \nu -\nu _{0}\right) \right) }{\left( i\pi
-x\right) ^{\alpha }}dx  \label{Int_Frac_J0_3} \\
&=&\Re \int_{0}^{\pi }e^{\pm i\nu x}\frac{P\left( \alpha ,\pm ix\left( \nu
-\nu _{0}\right) \right) }{\left( \pm ix\right) ^{\alpha }}dx.  \notag
\end{eqnarray}%
Insert (\ref{Int_Frac_J0_3})\ in (\ref{Int_frac_J0_1})\ to complete the
proof.
\end{proof}

\bigskip

\section{Conclusions\label{Section: Conclusions}}

On the one hand, we have obtained integral expressions for the $n$-th
derivatives of the Bessel functions with respect to the order in (\ref%
{Dn_Jnu}), (\ref{Dn_Ynu}), (\ref{Dn_Inu}), and (\ref{Dn_Knu}). Numerically,
these integral expressions are much more efficient than the numerical
evaluation of the corresponding derivatives for $n\geq 3$ using a \textit{double
exponential}\ integration strategy. As a by-product, we have calculated the
integral given in (\ref{int_new_G}), which does not seem to be reported in
the literature. Also, since the null value of $\partial ^{n}/\partial \nu
^{n}J_{\nu }\left( 0\right) $ and $\partial ^{n}/\partial \nu ^{n}I_{\nu
}\left( 0\right) $ is not trivial from the integral representations
obtained, a simple proof is included.

On the other hand, we have calculated the fractional integration of the
Bessel functions with respect to the order in (\ref{DJnu_a})-(\ref{DKnu_a}).
Also, we have included a simple proof to calculate the value of $D_{\nu -\nu
_{0}}^{-\alpha }J_{\nu }\left( 0\right) $ and $D_{\nu -\nu _{0}}^{-\alpha
}I_{\nu }\left( 0\right) $ from the corresponding integral representations
obtained.

Finally, in the Appendix, we provide simple proofs of the infinite integrals
of $J_{\nu }\left( t\right) $ and $I_{\nu }\left( t\right) $ with respect to
the order, i.e. (\ref{Appelblat_J})\ and (\ref{Appelblat_I}). In the
literature, these integrals are derived calculating the complex contour
integral of a particular inverse Laplace transform, but here we have derived
them by direct integration.

\appendix

\section{Derivation of (\protect\ref{Appelblat_J})}

Integrating\ with respect to the order in (\ref{Jnu_Int})\ and exchanging
the order of integration, we have 
\begin{eqnarray}
\int_{\nu }^{\infty }J_{\mu }\left( t\right) d\mu &=&\frac{1}{\pi }%
\int_{0}^{\pi }\int_{\nu }^{\infty }\cos \left( t\sin x-\mu x\right) d\mu
\,dx  \label{Resultado_0} \\
&&-\frac{1}{\pi }\int_{0}^{\infty }e^{-t\sinh x}\int_{\nu }^{\infty }e^{-\mu
x}\sin \mu \pi \,d\mu \,dx.  \notag
\end{eqnarray}

Notice that%
\begin{eqnarray*}
\int_{\nu }^{\infty }\cos \left( z\sin x-\mu x\right) d\mu
&=&\lim_{b\rightarrow \infty }\Re \left[ e^{it\sin x}\int_{\nu }^{b}e^{-i\mu
x}d\mu \right] \\
&=&\frac{1}{x}\lim_{b\rightarrow \infty }\Re \left[ i\,\left. e^{i\left(
t\sin x-\mu x\right) }\right\vert _{\mu =\nu }^{b}\right] .
\end{eqnarray*}%
\qquad

Since $\Re \left( iz\right) =-\Im \left( z\right) $ and $\lim_{b\rightarrow
\infty }\sin \left( t\sin x-bx\right) =-\lim_{b\rightarrow \infty }\sin bx$,
we have%
\begin{eqnarray}
\int_{\nu }^{\infty }\cos \left( z\sin x-\mu x\right) d\mu &=&\frac{-1}{x}%
\lim_{b\rightarrow \infty }\Im \left[ \left. e^{i\left( t\sin x-\mu x\right)
}\right\vert _{\mu =\nu }^{b}\right]  \label{Resultado_1} \\
&=&\lim_{b\rightarrow \infty }\frac{\sin bx}{x}+\frac{\sin \left( t\sin
x-\nu x\right) }{x}.  \notag
\end{eqnarray}

Also, considering that $x>0$,%
\begin{eqnarray}
\int_{\nu }^{\infty }e^{-\mu x}\sin \mu \pi \,d\mu &=&\Im \int_{\nu
}^{\infty }e^{\left( i\pi -x\right) \mu }d\mu  \label{Resultado_2} \\
&=&-\Im \left[ \frac{e^{\left( i\pi -x\right) \nu }}{i\pi -x}\right]  \notag
\\
&=&\frac{e^{-\nu x}}{\pi ^{2}+x^{2}}\left( \pi \cos \pi \nu +x\sin \pi \nu
\right) .  \notag
\end{eqnarray}

Substituting in (\ref{Resultado_0})\ the results (\ref{Resultado_1}) and (%
\ref{Resultado_2}), we obtain%
\begin{eqnarray}
\int_{\nu }^{\infty }J_{\mu }\left( z\right) d\mu &=&\frac{1}{\pi }%
\lim_{b\rightarrow \infty }\int_{0}^{\pi }\frac{\sin bx}{x}dx+\frac{1}{\pi }%
\int_{0}^{\pi }\frac{\sin \left( t\sin x-\nu x\right) }{x}dx
\label{Pre_Resultado} \\
&&-\frac{1}{\pi }\int_{0}^{\infty }\frac{e^{-t\sinh z-\nu x}}{\pi ^{2}+x^{2}}%
\left( \pi \cos \pi \nu +x\sin \pi \nu \right) dx.  \notag
\end{eqnarray}

Considering the definition of the sine integral \cite[Eqn. 6.2.9]{NIST}, we
have 
\begin{equation}
\frac{1}{\pi }\lim_{b\rightarrow \infty }\int_{0}^{\pi }\frac{\sin bx}{x}%
dx=\lim_{b\rightarrow \infty }\frac{\mathrm{Si}\left( b\pi \right) }{\pi }=%
\frac{1}{2},  \label{Resultado_3}
\end{equation}%
where we have applied \cite[Eqn. 6.2.14]{NIST},%
\begin{equation}
\lim_{x\rightarrow \infty }\mathrm{Si}\left( x\right) =\frac{\pi }{2}.
\label{Si_Lim}
\end{equation}

Therefore, substituting (\ref{Resultado_3})\ in (\ref{Pre_Resultado}), we
obtain the integral representation given in (\ref{Appelblat_J}).

\section{Derivation of (\protect\ref{Appelblat_I})}

Integrating\ with respect to the order in (\ref{Inu_int})\ and exchanging
the order of integration, we have%
\begin{eqnarray}
\int_{\nu }^{\infty }I_{\mu }\left( t\right) d\mu &=&\frac{1}{\pi }%
\int_{0}^{\pi }e^{t\cos x}\int_{\nu }^{\infty }\cos \mu x\,d\mu \,dx
\label{Resultado_I0} \\
&&-\frac{1}{\pi }\int_{0}^{\infty }e^{-t\cosh x}\int_{\nu }^{\infty }\sin
\pi \mu \,e^{-\mu x}d\mu \,dx.  \notag
\end{eqnarray}

Note that%
\begin{equation}
\int_{\nu }^{\infty }\cos \mu x\,d\mu =\lim_{b\rightarrow \infty }\frac{\sin
bx}{x}-\frac{\sin \nu x}{x}.  \label{Resultado_a}
\end{equation}

Therefore, substituting in (\ref{Resultado_I0}) the results (\ref%
{Resultado_3})\ and (\ref{Resultado_a}), we obtain%
\begin{eqnarray}
&&\int_{\nu }^{\infty }I_{\mu }\left( t\right) d\mu  \label{Pre_resultado_I}
\\
&=&\frac{1}{\pi }\lim_{b\rightarrow \infty }\int_{0}^{\pi }e^{t\cos x}\frac{%
\sin bx}{x}\,dx-\frac{1}{\pi }\int_{0}^{\pi }e^{t\cos x}\frac{\sin \nu x}{x}%
dx  \notag \\
&&-\frac{1}{\pi }\int_{0}^{\infty }\frac{e^{-t\cosh x-\nu x}}{\pi ^{2}+x^{2}}%
\left( \pi \cos \pi \nu +x\sin \pi \nu \right) dx.  \notag
\end{eqnarray}

We can calculate the above limit performing the substitution $u=bx$, 
\begin{eqnarray}
&&\frac{1}{\pi }\lim_{b\rightarrow \infty }\int_{0}^{\pi }e^{t\cos x}\frac{%
\sin bx}{x}\,dx  \notag \\
&=&\frac{1}{\pi }\lim_{b\rightarrow \infty }\int_{0}^{b\pi }e^{t\cos \left(
u/b\right) }\frac{\sin u}{u}\,du  \notag \\
&=&\frac{e^{t}}{\pi }\lim_{b\rightarrow \infty }\int_{0}^{b\pi }\frac{\sin u%
}{u}\,du  \notag \\
&=&\frac{e^{t}}{\pi }\lim_{b\rightarrow \infty }\mathrm{Si}\left( b\pi \right)
=\frac{e^{t}}{2},  \label{Resultado_b}
\end{eqnarray}%
where we have used (\ref{Si_Lim}). Therefore, inserting in (\ref%
{Pre_resultado_I})\ the result (\ref{Resultado_b}), we arrive at (\ref%
{Appelblat_I}).

\bibliographystyle{amsplain}

\end{document}